\documentclass[12pt]{amsart}
\usepackage{amssymb}
\usepackage[all]{xy}

\usepackage{verbatim}
\usepackage{eucal}
\usepackage{latexsym}
\usepackage{amsfonts,amssymb,amsxtra}
\usepackage{relsize}
\usepackage[dvips]{color}
\usepackage{latexsym}
\usepackage{mathptmx}
\usepackage{amssymb}
\usepackage{graphicx}
\usepackage{amsthm}
\usepackage{amsfonts}
\usepackage{amscd}
\usepackage{bm}
\usepackage[utf8]{inputenc}

\usepackage{verbatim} 
\usepackage{eucal} 
\usepackage{color}

\usepackage[draft]{todonotes}
\usepackage{pgf}
\usepackage{tikz}
\usetikzlibrary{arrows,automata}

\usepackage{pgf, tikz}
\definecolor {processblue}{cmyk}{0.96,0,0,0}
\input xy
\xyoption{all}
\usepackage{mathtools}
\usepackage{amsfonts}
\usepackage{amsthm}
\usepackage{mathrsfs}

\usepackage{amscd}
  \newtheorem{The}{Theorem}[section]
  \newtheorem{Pro}[The]{Proposition}
  \newtheorem{Lem}[The]{Lemma}

  \newtheorem{Rem}[The]{Remark}

\newcommand{\bsm}{\begin{smallmatrix}}
\newcommand{\esm}{\end{smallmatrix}}
\newcommand{\bbm}{\begin{matrix}}
\newcommand{\ebm}{\end{matrix}}

\newcommand{\Hom}{\rm{Hom}}

\newcommand{\End}{\rm{End}}

\theoremstyle{definition}

\theoremstyle{plain}

\theoremstyle{definition}

\numberwithin{equation}{section}
\newtheorem*{theorem a*}{Theorem A}
\newtheorem*{theorem b*}{Theorem B}
\begin{document}

\title{ Auslander correspondence for kawada rings}

\author{Ziba Fazelpour}
\address{School of Mathematics, Institute for Research in Fundamental Sciences (IPM), P.O. Box: 19395-5746, Tehran, Iran}
\email{z.fazelpour@ipm.ir}
\author{Alireza Nasr-Isfahani}
\address{Department of Pure Mathematics\\
Faculty of Mathematics and Statistics\\
University of Isfahan\\
P.O. Box: 81746-73441, Isfahan, Iran\\ and School of Mathematics, Institute for Research in Fundamental Sciences (IPM), P.O. Box: 19395-5746, Tehran, Iran}
\email{nasr$_{-}$a@sci.ui.ac.ir / nasr@ipm.ir}



\begin{abstract}
We study the Auslander ring of a basic left K\"othe ring $\Lambda$ and give a characterization of basic left K\"othe rings in terms of their Auslander rings. We also study the functor category Mod$((\Lambda$-mod$)^{\rm op})$ and characterize basic left K\"othe rings $\Lambda$ by using functor categories Mod$((\Lambda$-mod$)^{\rm op})$. As a consequence we show that there exists a bijection between the Morita equivalence classes of left Kawada rings and the Morita equivalence classes of Auslander generalized right QF-2 rings.
\end{abstract}

\maketitle

\section{Introduction}

A ring $\Lambda$ is said to be {\it of finite representation type} or {\it representation-finite} if it is left artinian and there are only finitely many finitely generated indecomposable  left $\Lambda$-modules up to isomorphism. The class of representation-finite rings is one of the most important class of rings. Auslander in his famous theorem, which is called Auslander correspondence, provided a bijection between the set of Morita equivalence classes of representation-finite rings and that of rings with nice homological properties. Let $\Gamma$ be an artinian ring. The {\it global dimension} of $\Gamma$, which is denoted by ${\rm gl.dim}(\Gamma)$, is at most $n$ if the projective dimension of each left $\Gamma$-module is less than or equal $n$. The dominant dimension of $\Gamma$, which is denoted by ${\rm dom.dim}(\Gamma)$, is at least $n$ if the first $n$ terms in the minimal injective resolution of $_{\Gamma}\Gamma$ are projectives. An artinian ring $\Gamma$ is called an {\it Auslander ring} if  ${\rm gl.dim}(\Gamma) \leq 2$ and ${\rm dom.dim}(\Gamma) \geq 2$. According to the Auslander correspondence  \cite[Corollary 4.7]{ausla2}, there exists a bijection between the set of Morita equivalence classes of representation-finite rings and that of Auslander rings $\Gamma$. It is given by $\Lambda\rightarrow\Gamma:=\End_{\Lambda}(M)$, where $M=\bigoplus_{i=1}^nM_i$ and $\lbrace M_1, \cdots, M_n\rbrace$ is a complete set of representative of the isomorphic classes of finitely generated indecomposable left $\Lambda$-modules. $M$ is called the Auslander generator of $\Lambda$-mod. Auslander correspondence gives a very nice connection between concept of representation-finiteness which is a representation theoretic property and the concept of Auslander ring which is a homological property. On the other hand, the category Mod-$\Gamma$ is equivalent to the functor category Mod$((\Lambda$-mod$)^{\rm op})$, where Mod$((\Lambda$-mod$)^{\rm op})$ is the category of all additive covariant functors from $\Lambda$-mod to the category of abelian groups. This equivalence propose to use the functor categories for studying the representation-finite rings which leads
to functorial approach in representation theory (see \cite{Au, ausla2}).
A functor $F\in$ Mod$((\Lambda$-mod$)^{\rm op})$ is called noetherian (resp., artinian) if it satisfies the
ascending (resp., descending) chain condition on subfunctors. $F$ is called finite if it is both noetherian and artinian and $F$ is called locally finite if every finitely generated
subfunctor of F is finite. The category  Mod$((\Lambda$-mod$)^{\rm op})$ is said to be locally finite if every $F\in$ Mod$((\Lambda$-mod$)^{\rm op})$ is locally finite. Auslander in Theorem 3.1 of \cite{ausla2} proved that a ring $\Lambda$ is of finite representation type if and only if Mod$((\Lambda$-mod$)^{\rm op})$ is locally finite if and only if every representable functor of Mod$((\Lambda$-mod$)^{\rm op})$ is both artinian and noetherian (see also \cite{aus}).

It is known that every finitely generated ${\Bbb Z}$-module decompose to the finite direct sum of cyclic modules. Artinian rings with this property are representation-finite rings. These rings are precisely rings over which every module is a direct sum of cyclic modules. A ring $\Lambda$ is called {\it left K\"othe} if every left $\Lambda$-module is a direct sum of cyclic modules. K\"othe studied this class of rings which is an important subclass of representation-finite rings. He showed that commutative artinian rings which have this property are serial. Also he posed the question to classify the non-commutative rings with this property \cite{K} (see also \cite{Jain} and references therein). Kawada completely solved the K\"othe's problem for the basic finite dimensional $K$-algebras. Kawada's papers contain a set of 19 conditions which characterize Kawada algebras, as well as, the list of all possible finitely generated indecomposable modules \cite{Kawada1, Kawada2, Kawada3}.  A ring $\Lambda$ is called {\it left Kawada} if any ring Morita equivalent to $\Lambda$ is a left K\"othe ring \cite{r}. Ringel showed that any finite dimensional $K$-algebra of finite representation type is Morita equivalent to a K\"othe algebra. By using of the multiplicity-free of top and soc of finitely generated indecomposable modules, he also gave a characterization of Kawada algebras \cite{r}. In \cite{za} the authors proved the Ringel's result for artinian rings.

Since any left K\"othe ring is representation-finite, the above results of Auslander propose to use the Auslander ring and the functor category Mod$((\Lambda$-mod$)^{\rm op})$ for study left K\"othe rings. In this paper we study left K\"othe rings by using their Auslander rings and their functor categories.
We say that a semiperfect ring with enough idempotents $R$ is {\it generalized left {\rm (}resp., right{\rm )} QF-2} if every indecomposable projective unitary left (resp., right) $R$-module $P$ has multiplicity-free socle (i.e. the composition factors of soc$(P)$ are pairwise non-isomorphic). In the following theorem, by using the notion of generalized right QF-2 rings, we give a characterization of basic left K\"othe rings.

\newtheorem{thm}{Theorem}
\begin{theorem a*}$($See Theorem \ref{r33}$)$
Let $\Lambda$ be a basic ring. Then the following conditions are equivalent.
\begin{itemize}
\item[$(a)$] $\Lambda$ is a left K\"othe ring.
\item[$(b)$] $\Lambda$ is a ring of finite representation type and the Auslander ring $T$ of $\Lambda$ is a generalized right QF-2 ring.
\item[$(c)$] $\Lambda$ is a left artinian ring and $T=\widehat{{\rm End}}_{\Lambda}(V)$ is a left locally finite generalized right QF-2 ring, where $V$ is a direct sum of modules in a complete set of representative of the isomorphic classes of  finitely generated indecomposable left $\Lambda$-modules.
\item[$(d)$] $\Lambda$ is a left artinian ring and for each indecomposable left $\Lambda$-module $M$, ${\rm Hom}_{\Lambda}(M,-) \simeq \varphi(-)$ for some positive-primitive formula $\varphi(x)$ over $\Lambda$.
\item[$(e)$] $\Lambda$ is a left artinian ring and every indecomposable projective object in  Mod$((\Lambda$-mod$)^{\rm op})$ has finitely generated essential multiplicity-free socle.
\end{itemize}
\end{theorem a*}

Using the above theorem, we prove the following theorem, which gives a characterization of left Kawada rings.

\begin{theorem b*}$($See Theorem \ref{ka}$)$
There exists a bijection between the Morita equivalence classes of left Kawada rings and the Morita equivalence classes of Auslander generalized right QF-2 rings.
\end{theorem b*}

Before proving our main results in Section 3 we prove some preliminary results in the following section.
\subsection{Notation }
Throughout this paper all rings are associative with unit unless otherwise stated. Let $R$ be a ring (not necessary with unit). We write all homomorphisms of left (resp., right) $R$-modules on the right (resp., left), so $fg$ (resp., $g \circ f$) means ``first $f$ then $g$".  We denote by $R$-Mod (resp., Mod-$R$) the category of all left (resp., right) $R$-modules and by $J(R)$ the Jacobson radical of $R$.  Also we denote by $R$-mod (resp., mod-$R$) the category of all finitely generated left (resp., right) $R$-modules. A left (resp., right) $R$-module $M$ is called {\it unitary} if $RM=M$ (resp., $MR=M$). We denote by $R$Mod (resp., Mod$R$) the category of all unitary left (resp., right) $R$-modules. We denote by Mod$((R$-mod$)^{\rm op})$ the category of all additive covariant functors from $R$-mod to the category of abelian groups, which is an abelian category. We write $\bigoplus_A M$ and $\prod_A M$ for the direct sum of card$(A)$ copies of an $R$-module $M$ and the direct product of card$(A)$ copies of an $R$-module $M$, respectively. For a left $R$-module $V$ we denoted by ${\rm Add}(V)$ the full subcategory of $R$-Mod whose objects are all left $R$-modules that are isomorphic to direct summands of $\bigoplus_A V$ for any set $A$. We denoted by ${\rm Proj}(R)$ (resp., ${\rm Proj}(R^{\rm op})$) the full subcategory of $R$Mod (resp., Mod$R$) whose objects are projective left (resp., right) $R$-modules. Let $N$ be a unitary left $R$-module. We denote by $E(N)$, ${\rm rad}(N)$, ${\rm top}(N)$ and ${\rm soc}(N)$ the injective hull of $N$ in $R$Mod, radical of $N$, top of $N$ and socle of $N$, respectively. Let $f : X \rightarrow Y$ be an $R$-module homomorphism and assume that $M$ is an $R$-submodule of $X$. We denote by $f\mid_M$ the restriction of $f$ to $M$. Let  $L$ and $N$ be two left $R$-modules. The submodule ${\rm Re}(L, N)=\bigcap \lbrace {\rm Ker}f~|~ f\in {\rm Hom}_{R}(L,N) \rbrace$ is called {\it reject of $N$ in $L$}.

\section{Preliminaries}

 Let $\mathcal{U}$ be a non-empty set of left $\Lambda$-modules. A left $\Lambda$-module $M$ is said to be {\it generated} by $\mathcal{U}$ if for every pair of distinct morphisms $f,g : M \rightarrow B$ in $\Lambda$-Mod there exists a morphism $h : U \rightarrow M$ with $U \in \mathcal{U}$ such that  $h f \neq hg$. Also, $\mathcal{U}$ is called a {\it generating set} for $\Lambda$-Mod if every left $\Lambda$-module generated by $\mathcal{U}$. A left $\Lambda$-module $U$ is called {\it generator} in $\Lambda$-Mod if set $\mathcal{U}=\lbrace U \rbrace$  is a generating set for $\Lambda$-Mod (see \cite[Ch. V, Sect. 7]{s}). Let $\Lambda$ be a ring of finite representation type and $\lbrace X_1, \cdots, X_n\rbrace$ be a complete set of representative of the isomorphic classes of finitely generated indecomposable left $\Lambda$-modules. We recall that $X=\bigoplus_{i=1}^nX_i$ is called an Auslander generator of $\Lambda$-mod. Also the endomorphism ring of $X$ is called {\it Auslander ring of} $\Lambda$. Note that by \cite[Proposition 3.6]{ausla2}, the Auslander ring of $\Lambda$ is an artinian ring. \\

Let $\Lambda$ be a ring and $\lbrace V_{\alpha}~|~ \alpha \in J \rbrace$ be a family of finitely generated left $\Lambda$-modules. Set $V=\bigoplus_{\alpha \in J} V_{\alpha}$ and for each $\alpha \in J$, letting $e_{\alpha}=\pi_{\alpha}\varepsilon_{\alpha}$, where $\pi_{\alpha}:V \rightarrow V_{\alpha}$ is the canonical projection and $\varepsilon_{\alpha}:V_{\alpha} \rightarrow V$ is the canonical injection. For each left $\Lambda$-module $X$, we define as in \cite[Page 40]{fu}, $\widehat{{\Hom}}_{\Lambda}(V,X)=\lbrace f \in {\rm Hom}_{\Lambda}(V,X)~|~(V)e_{\alpha} f=0~ {\rm for ~almost~ all}~ \alpha \in J \rbrace$. For $X=V$, we write $\widehat{{\Hom}}_{\Lambda}(V,V)=\widehat{{\rm End}}_{\Lambda}(V)$.  Let $R$ be a ring (not necessary with unit). $R$ is called a {\it ring with enough idempotents} if there exists a family $\lbrace q_{\alpha}~|~\alpha \in I \rbrace$ of pairwise orthogonal idempotents of $R$ with $R=\bigoplus_{\alpha \in I}Rq_{\alpha}=\bigoplus_{\alpha \in I}q_{\alpha}R$ (see \cite[Page 39]{fu}).  $T=\widehat{{\rm End}}_{\Lambda}(V)$ is a ring with enough idempotents because of $T=\bigoplus_{\alpha \in J}Te_{\alpha}=\bigoplus_{\alpha \in J}e_{\alpha}T$. Fuller in \cite[Page 40]{fu} defined a covariant functor $\widehat{{\Hom}}_{\Lambda}(V,-):\Lambda$-Mod$\rightarrow T$Mod as follows. For any morphism $f: X \rightarrow Y$ in $\Lambda$-Mod, he defined $\widehat{{\rm Hom}}_{\Lambda}(V,f): \widehat{{\rm Hom}}_{\Lambda}(V,X) \rightarrow \widehat{{\rm Hom}}_{\Lambda}(V,Y)$ via $g \mapsto gf$. From \cite[Pages 40-41]{fu} we observe that the covariant functor $\widehat{{\Hom}}_{\Lambda}(V,-)$ is a left exact functor and also preserves direct sums.  Moreover it is an additive equivalence between the full subcategory ${\rm Add}(V)$ of $\Lambda$-Mod and the full subcategory Proj$(T)$ of $T$Mod with the inverse equivalence $V \otimes_T -$.  Note that the covariant functor $\widehat{{\Hom}}_{\Lambda}(V,-):\Lambda$-Mod$\rightarrow T$Mod preserves indecomposable modules, injective modules and essential extensions when $V$ is a generator in $\Lambda$-Mod (see \cite[Proposition 51.7]{wi}). Also it is easily seen that $V$ is a generator in $\Lambda$-Mod when $\Lambda$ is left artinian and $\lbrace V_{\alpha}~|~ \alpha \in J \rbrace$ is a complete set of representative of the isomorphic classes of finitely generated indecomposable left $\Lambda$-modules. By using \cite[Proposition 4.2 and Corollary 4.8]{ausla2} and \cite[Proposition 46.7]{wi} we have the following result.

\begin{Lem}\label{p0}
Let $\Lambda$ be a ring of finite representation type, $X$ be an Auslander generator and $T={\rm End}_{\Lambda}(X)$. Then
\begin{itemize}
\item[$(a)$] The functor ${\rm Hom}_{\Lambda}(X,-): \Lambda$-{\rm Mod}$ \rightarrow${\rm Proj}$(T)$ is an equivalence which preserves and reflects finitely generated modules.
\item[$(b)$] The functor ${\rm Hom}_{\Lambda}(X,-)$ preserves essential extension.
\item[$(c)$] If $E$ is an injective left $\Lambda$-module, then ${\rm Hom}_{\Lambda}(X,E)$ is an injective left $T$-module.
\end{itemize}
\end{Lem}

A ring $\Lambda$ is called {\it semiperfect} if $\Lambda/J(\Lambda)$ is a semisimple ring and idempotents in $\Lambda/J(\Lambda)$ can be lifted to $\Lambda$ (see \cite[Page 303]{an}). It is well-known that a ring with enough idempotents $R$ is semiperfect if and only if every finitely generated unitary left $R$-module has a projective cover in $R$Mod if and only if every finitely generated projective unitary left $R$-module is a direct sum of local modules (see \cite[Page 95]{f} and \cite[Proposition 49.10]{wi}).\\

A semiperfect ring $\Lambda$ is called {\it left {\rm (}resp., right{\rm )} QF-2} if every indecomposable projective left (resp., right) $\Lambda$-module has a simple essential socle (see \cite[Sect. 4]{f}).  A left $\Lambda$-module $M$ has {\it multiplicity-free socle} if composition factors of soc$(M)$ are pairwise non-isomorphic (see \cite[Sect. 1]{r}).  We say that a semiperfect ring with enough idempotents $R$ is  {\it generalized left {\rm (}resp., right{\rm )} QF-2} if every indecomposable projective unitary left (resp., right) $R$-module has multiplicity-free socle.

\begin{Pro}\label{r1}
Let $\Lambda$ be a ring of finite representation type.  Then the following conditions are equivalent.
\begin{itemize}
\item[$(a)$] The Auslander ring of $\Lambda$ is a generalized left QF-2 ring.
\item[$(b)$] Every indecomposable left $\Lambda$-module has multiplicity-free socle.
\end{itemize}
\end{Pro}
\begin{proof}
Let $X$ be an Auslander generator of $\Lambda$-mod and $T = {\rm End}_{\Lambda}(X)$. Assume that $M$ is a finitely generated left $\Lambda$-module. Since $\Lambda$ is left artinian, by \cite[Propositions 21.3 and 31.4]{wi}, soc$(M)$ is a finitely generated and essential submodule of $M$. It is easy to see that soc$(M)=\bigoplus_{i=1}^nT_i$, where $n \in {\Bbb{N}}$ and each $T_i$ is a simple left $\Lambda$-module. Since soc$(M)$ is an essential submodule of $M$,  we get $E(M) \cong \bigoplus_{i=1}^n E(T_i)$ as $\Lambda$-modules. On the other hand, by Lemma \ref{p0}(b), there exists an essential monomorphism  ${\rm Hom}_{\Lambda}(X,M) \rightarrow {\rm Hom}_{\Lambda}(X,E(M))$ as $T$-modules. Also by Lemma \ref{p0}(c), we see that ${\rm Hom}_{\Lambda}(X,E(M))$ is an injective left $T$-module. Therefore $$E({\rm Hom}_{\Lambda}(X,M)) \cong {\rm Hom}_{\Lambda}(X,E(M)) \cong \bigoplus_{i=1}^n {\rm Hom}_{\Lambda}(X,E(T_i))$$ as $T$-modules. ${\rm End}_{T}({\rm Hom}_{\Lambda}(X,E(T_i)))$ is a local ring because the functor
\begin{center}
${\rm Hom}_{\Lambda}(X,-):\Lambda$-Mod$ \rightarrow$Proj$(T)$
\end{center}
is an equivalence and ${\rm End}_{\Lambda}(E(T_i))$ is a local ring. \\
$(a)\Rightarrow (b)$. Let $M$ be an indecomposable left $\Lambda$-module.  Since $\Lambda$ is a ring of finite representation type, by \cite[Corollary 4.8]{ausla2}, $M$ is finitely generated. Hence soc$(M)=\bigoplus_{i=1}^nT_i$, where $n \in {\Bbb{N}}$ and each $T_i$ is a simple left $\Lambda$-module. We show that $T_i \ncong T_j$ when $i \neq j$.  By Lemma \ref{p0}(a), ${\rm Hom}_{\Lambda}(X,M)$ is a finitely generated indecomposable projective left $T$-module. Since $T$ is an artinian ring,  ${\rm soc}({\rm Hom}_{\Lambda}(X,M))= \bigoplus_{i=1}^tS_i$, where $t \in {\Bbb{N}}$, each $S_i$ is a simple left $T$-module.  Moreover  $E({\rm Hom}_{\Lambda}(X,M)) \cong \bigoplus_{i=1}^tE(S_i)$ as $T$-modules. By assumption, $S_i \ncong S_j$ when $i \neq j$. On the other hand, by above statement $E({\rm Hom}_{\Lambda}(X,M)) \cong \bigoplus_{i=1}^n {\rm Hom}_{\Lambda}(X,E(T_i))$ as $T$-modules, where $n \in {\Bbb{N}}$, each $T_i$ is a simple left $\Lambda$-module and soc$(M)=\bigoplus_{i=1}^nT_i$. Also ${\rm End}_{T}({\rm Hom}_{\Lambda}(X,E(T_i)))$ is a local ring for each $1 \leq i \leq n$. Therefore  $t=n$ and for each $1 \leq i \leq n$, $E(S_i) \cong {\rm Hom}_{\Lambda}(X,E(T_i))$ as $T$-modules. We can certainly conclude that  $T_i \ncong T_j$ when $i \neq j$ since otherwise  $E(T_i) \cong E(T_j)$ as $\Lambda$-modules for some $i \neq j$. This implies that ${\rm Hom}_{\Lambda}(X,E(T_i)) \cong {\rm Hom}_{\Lambda}(X,E(T_j))$ as $T$-modules. $E(S_i) \cong E(S_j)$ yields $S_i \cong S_j$ which is a contradiction.  Therefore every indecomposable left $\Lambda$-module has multiplicity-free socle.\\
$(b)\Rightarrow (a)$. Let $P$ be an indecomposable projective left $T$-module. Then $P$ is a finitely generated left $T$-module because $T$ is an artinian ring.  Hence by Lemma \ref{p0}(a), $P \cong {\rm Hom}_{\Lambda}(X,M)$ for some a finitely generated indecomposable left $\Lambda$-module $M$. Then $E(P) \cong E({\rm Hom}_{\Lambda}(X,M)) \cong \bigoplus_{i=1}^n {\rm Hom}_{\Lambda}(X,E(T_i))$ and ${\rm End}_{T}({\rm Hom}_{\Lambda}(X,E(T_i)))$ is a local ring for each $1 \leq i \leq n$, where $n \in {\Bbb{N}}$, ${\rm soc}(M)=\bigoplus_{i=1}^n T_i$ and each $T_i$ is a simple left $\Lambda$-module. On the other hand, ${\rm soc}(P) = \bigoplus_{i=1}^s S_i$ and $E(P) \cong \bigoplus_{i=1}^s E(S_i)$, where $s \in {\Bbb{N}}$ and each $S_i$ is a simple left $T$-module. Therefore $s=n$ and $E(S_i) \cong {\rm Hom}_{\Lambda}(X,E(T_i))$ for each $1 \leq i \leq n$. If $S_i \cong S_j$ as $T$-modules for some $i \neq j$, then $E(S_i) \cong E(S_j)$ as $T$-modules and hence  ${\rm Hom}_{\Lambda}(X,E(T_i)) \cong {\rm Hom}_{\Lambda}(X,E(T_j))$ as $T$-modules. So by Lemma \ref{p0}(a), we have $E(T_i) \cong E(T_j)$ as $\Lambda$-modules. It follows that $T_i \cong T_j$ as $\Lambda$-modules which is a contradiction. Consequently $T$ is a generalized left QF-2 ring and the result follows.
\end{proof}

Remark \ref{se} and Lemma \ref{l1} are well-known to experts, but since we could not find a suitable reference in the literature, we include these results here.

\begin{Rem}\label{se}
{\rm  Let $\Lambda$ be a left artinian ring and $\lbrace V_{\alpha}~|~\alpha \in J \rbrace$ be a family of finitely generated indecomposable left $\Lambda$-modules. Set $V=\bigoplus_{\alpha \in J}V_{\alpha}$. Let $\pi_{\alpha}: V \rightarrow V_{\alpha}$ be the canonical projection and $\varepsilon_{\alpha}: V_{\alpha} \rightarrow V$ be the canonical injection. Set $e_{\alpha}=\pi_{\alpha}\varepsilon_{\alpha}$ for each $\alpha \in J$. Then $T=\widehat{{\rm End}}_{\Lambda}(V) =\bigoplus_{\alpha \in J}Te_{\alpha}=\bigoplus_{\alpha \in J}e_{\alpha}T$. Since $\widehat{{\Hom}}_{\Lambda}(V,-): {\rm Add}(V) \rightarrow {\rm Proj}(T)$ is an equivalence of categories and $Te_{\alpha} \cong \widehat{{\rm Hom}}_{\Lambda}(V, V_{\alpha})$ as $T$-modules we have for each $\alpha \in J$,  $e_{\alpha}Te_{\alpha} \cong {\rm End}_{\Lambda}(V_{\alpha})$ as rings.  $V_{\alpha}$ has finite length and also it is indecomposable, then by \cite[Proposition 32.4(3)]{wi}, $e_{\alpha}Te_{\alpha}$ is a local ring. It follows that every finitely generated projective unitary left $T$-module is a direct sum of local modules. Consequently $T$ is a semiperfect ring.}
\end{Rem}

\begin{Lem}\label{l1}
Let $\Lambda$ be a ring and assume that $\lbrace V_{\alpha}~|~\alpha \in I \rbrace$ is a family of non-isomorphic finitely generated indecomposable left $\Lambda$-modules. Set $T=\widehat{{\rm End}}_{\Lambda}(V)$, where $V=\bigoplus_{\alpha \in I}V_{\alpha}$. If $\Lambda$ is a left artinian ring, then  $\lbrace Te_{\alpha}/{\rm rad}(Te_{\alpha})~|~\alpha \in I \rbrace$ is a complete set of representative of the isomorphic classes of  simple unitary left $T$-modules, where $e_{\alpha}=\pi_{\alpha}\varepsilon_{\alpha}$, $\pi_{\alpha}: V \rightarrow V_{\alpha}$ is the canonical projection and $\varepsilon_{\alpha}: V_{\alpha} \rightarrow V$ is the canonical injection.
\end{Lem}
\begin{proof}
By Remark \ref{se}, ${\rm End}_T(Te_{\alpha})$ is a local ring.  $Te_{\alpha}$ is a projective unitary left $T$-module and hence ${\rm rad}(Te_{\alpha})$ is a superfluous maximal submodule of $Te_{\alpha}$. Then $Te_{\alpha}/{\rm rad}(Te_{\alpha})$ is a simple unitary left $T$-module. Also the canonical projection $Te_{\alpha} \rightarrow Te_{\alpha}/{\rm rad}(Te_{\alpha})$ is a projective cover of $Te_{\alpha}/{\rm rad}(Te_{\alpha})$ in $T$Mod. If $Te_{\alpha}/{\rm rad}(Te_{\alpha}) \cong Te_{\beta}/{\rm rad}(Te_{\beta})$ as $T$-modules, then $Te_{\alpha} \cong Te_{\beta}$. It is easy to see that $V_{\alpha} \cong V_{\beta}$ as $\Lambda$-modules and hence $\alpha=\beta$. Therefore $\lbrace Te_{\alpha}/{\rm rad}(Te_{\alpha})~|~\alpha \in I \rbrace$ is a set of non-isomorphic simple unitary left $T$-modules. It remains to prove that  every simple unitary left $T$-module $S$ is isomorphism to $Te_{\alpha}/{\rm rad}(Te_{\alpha})$ for some $\alpha \in I$. Let $S$ be a simple unitary left $T$-module. By Remark \ref{se}, $T$ is a semiperfect ring and hence $S$ has a projective cover in $T$Mod.  Let $f: P \rightarrow S$ be a projective cover of $S$ in $T$Mod. Since ${\rm Ker}f$ is a maximal superfluous submodule of $P$, $P$ is an indecomposable unitary left $T$-module. It follows that $P \cong Te_{\alpha}$ as $T$-modules for some $\alpha \in I$ and hence $S \cong Te_{\alpha}/{\rm rad}(Te_{\alpha})$ as $T$-modules and the result follows.
\end{proof}

Let $\Lambda$ be a ring and $\lbrace V_{\alpha}~|~ \alpha \in J \rbrace$ be a family of finitely generated left $\Lambda$-modules. Set $V=\bigoplus_{\alpha \in J} V_{\alpha}$ and $R=\widehat{{\rm End}}_{\Lambda}(V)$. As in \cite[Page 40]{fu}, we call $R$  {\it the left functor ring of $\Lambda$} when $\lbrace V_{\alpha}~|~ \alpha \in J \rbrace$ is a complete set of representative of the isomorphic classes of finitely generated left $\Lambda$-modules. Two rings  with enough idempotents $R$ and $S$ are said to be {\it Morita equivalent} if there exists an additive covariant equivalence between the category of unitary left $R$-modules and the category of unitary left $S$-modules (see \cite[Sect. 3]{a}).

\begin{Pro}\label{r3}
Let $\Lambda$ be a ring and $\lbrace V_{\alpha}~|~ \alpha \in J \rbrace$ be a complete set of representative of the isomorphic classes of finitely generated indecomposable left $\Lambda$-modules. Set $V=\bigoplus_{\alpha \in J} V_{\alpha}$ and $T=\widehat{{\rm End}}_{\Lambda}(V)$. If every left $\Lambda$-module is a direct sum of finitely generated modules, then $T$ is Morita equivalent to the left functor ring $R$ of $\Lambda$.
\end{Pro}
\begin{proof}
First we show that ${\rm Proj}(R)$ and ${\rm Proj}(T)$ are equivalent. Let $\lbrace U_{\alpha}~|~ \alpha \in I \rbrace$ be a complete set of representative of the isomorphic classes of finitely generated left $\Lambda$-modules such that $R=\widehat{{\rm End}}_{\Lambda}(U)$, where $U= \bigoplus_{\alpha \in I}U_{\alpha}$.  By assumption  Add$(U)=\Lambda$-Mod and also by  \cite[Theorem 4.4]{chase}, we see that $\Lambda$ is a left artinian ring. Since  Add$(U)=\Lambda$-Mod,  the functor  $\widehat{{\rm Hom}}_{\Lambda}(U,-): \Lambda$-Mod$ \rightarrow {\rm Proj}(R)$ is an additive equivalence with the inverse equivalence $U\otimes_R-$. On the other hand,  every finitely generated left $\Lambda$-module is a direct sum of indecomposable modules because $\Lambda$ is left artinian. Then Add$(V)=\Lambda$-Mod and we have an additive equivalence $\widehat{{\rm Hom}}_{\Lambda}(V,-): \Lambda$-Mod$ \rightarrow {\rm Proj}(T)$ with the inverse equivalence $V\otimes_T-$. Therefore we obtain an additive equivalence $\widehat{{\rm Hom}}_{\Lambda}(V,U\otimes_R-): {\rm Proj}(R) \rightarrow {\rm Proj}(T)$ with the inverse equivalence  $\widehat{{\rm Hom}}_{\Lambda}(U,V\otimes_T-): {\rm Proj}(T) \rightarrow {\rm Proj}(R)$. Now we show that this equivalence preserves finitely generated modules. Let $Y$ be a finitely generated projective unitary left $R$-module. Set $X=U\otimes_RY$. Since $\widehat{{\rm Hom}}_{\Lambda}(U,X) \cong Y$ and $\widehat{{\rm Hom}}_{\Lambda}(U,-)$ preserves direct sums, $X$ is a finitely generated left $\Lambda$-module and hence $X$ is a finite direct sum of finitely generated indecomposable left $\Lambda$-modules. Then $\widehat{{\rm Hom}}_{\Lambda}(V,X)$ is a finitely generated projective unitary left $T$-module. Consequently the equivalence $\widehat{{\rm Hom}}_{\Lambda}(V,U\otimes_R-): {\rm Proj}(R) \rightarrow {\rm Proj}(T)$  preserves finitely generated projective modules. By the similar argument, we can see that it also reflects finitely generated projective modules. Therefore the result follows from \cite[Theorem 3.10]{zn}.
\end{proof}

Let $\Lambda$ be a ring and $\mathcal{U}$ be a non-empty set of left $\Lambda$-modules. A left $\Lambda$-module $M$ is said to be {\it cogenerated by} $\mathcal{U}$ if for every pair of distinct morphisms $f,g : B \rightarrow M$ in $\Lambda$-Mod, there exists a morphism $h : M \rightarrow U$ with $U \in \mathcal{U}$ such that $fh \neq gh$. Also, $\mathcal{U}$ is called a {\it cogenerating set} for $\Lambda$-Mod if every left $\Lambda$-module cogenerated by $\mathcal{U}$. A left $\Lambda$-module $U$ is called {\it cogenerator} if the set $\mathcal{U}=\lbrace U \rbrace$ is a cogenerating set for $\Lambda$-Mod (see \cite[Ch. V, Sect. 7]{s}). A ring $\Lambda$ is called {\it left Morita} if there is an injective cogenerator left $\Lambda$-module $Q$ such that it is an injective cogenerator right $\Delta$-module and $\Lambda \cong {\rm End}_{\Delta}(Q)$, where $\Delta = {\rm End}_{\Lambda}(Q)$ (see \cite[Ch. 9, Sect. 47]{wi}). In this case we say that $\Lambda$ is left Morita to $\Delta$.

\begin{Pro}\label{p6}
Let $\Lambda$ be a ring and $R$ be the left functor ring of $\Lambda$. If $\Lambda$ is a left artinian ring and every finitely generated indecomposable projective unitary right $R$-module has finitely generated essential socle, then $\Lambda$ is a left Morita ring.
\end{Pro}
\begin{proof}
By \cite[Theorem 3.3]{f} it is enough to show that ${\rm soc}(R_R)$ is essential in $R_R$ and also $R_R$ has only finitely many non-isomorphic simple right $R$-modules. Suppose that $\lbrace U_{\alpha}~|~\alpha \in I \rbrace$ is a complete set of representative of the isomorphic classes of finitely generated left $\Lambda$-modules and $R=\widehat{{\rm End}}_{\Lambda}(U)$, where $U=\bigoplus_{\alpha \in I}U_{\alpha}$. $U$ is a generator in $\Lambda$-Mod. Since $\Lambda$ is left artinian, by \cite[Proposition 51.7(9)]{wi}, $R$ is a semiperfect ring. Therefore  every finitely generated projective unitary right $R$-module is a direct sum of indecomposable modules. Hence ${\rm soc}(R_R)$ is an essential submodule of $R_R$. Now we show that $R_R$ has only finitely many non-isomorphic simple right $R$-modules. $_{\Lambda}\Lambda \cong U_{\alpha}$ for some $\alpha \in I$ and so  $U \cong e_{\alpha}R$ as $R$-modules, where $e_{\alpha}=\pi_{\alpha}\varepsilon_{\alpha}$, $\pi_{\alpha}:U \rightarrow U_{\alpha}$ is the canonical projection and $\varepsilon_{\alpha}:U_{\alpha} \rightarrow U$ is the canonical injection. Moreover $e_{\alpha}R = \bigoplus_{i=1}^{t_{\alpha}} P_i$, where each $P_i$ is an indecomposable right $R$-module and hence ${\rm soc}(e_{\alpha}R) \cong \bigoplus_{i=1}^{t_{\alpha}}{\rm soc}(P_i)$ as $R$-modules. Since every finitely generated indecomposable projective unitary right $R$-module has finitely generated socle, ${\rm soc}(e_{\alpha}R)$ is finitely generated. In order to complete the proof, it is sufficient to show that for each simple right $R$-submodule $L$ of $R$, there is an $R$-module monomorphism $L \rightarrow e_{\alpha}R$. Since $e_{\alpha}R$ is faithful,  ${\rm Re}(R_R, e_{\alpha}R)=0$.  It follows that by \cite[Proposition 14.3]{wi}, there exists an $R$-module monomorphism $\varphi: R_R \rightarrow \prod_A e_{\alpha}R$.  Let $L$ be a simple submodule of $R_R$. Then $L=aR$ for some $0 \neq a \in R$. Put $\varphi(a)=(y_i)$, where each $y_i \in e_{\alpha}R$. Since $\varphi$ is a monomorphism,  $y_i \neq 0$ for some $i$. Without loss of generality we can assume that $y_1 \neq 0$. Since $R$ is a ring with enough idempotents, there exists an idempotent $e \in R$ such that $ae=a$ and $y_1e=y_1$. Therefore the $R$-module homomorphism $\pi_1 \circ \overline{\varphi}: L \rightarrow e_{\alpha}R$ is a non-zero morphism, where $\overline{\varphi}:= \varphi|_{L}$ and $\pi_1: \prod_A e_{\alpha}R \rightarrow e_{\alpha}R$ is the canonical projection and the result follows.
\end{proof}

A left $\Lambda$-module $W$ is called {\it minimal cogenerator} if $W \cong \bigoplus_{i\in I} E(S_i)$, where $\lbrace S_i~|~i\in I\rbrace$ is a complete set of representative of the isomorphic classes of simple left $\Lambda$-modules (see \cite[Ch. 3, Sect. 17]{wi}). Let $\Lambda$ and $\Delta$ be two rings. An additive contravariant functor $F : \Lambda$-mod $\rightarrow$ mod-$\Delta$ is called duality if it is an equivalence of categories (see \cite[Ch. 9, Sect. 47]{wi}).

\begin{Rem}\label{r2}{\rm
Let $\Lambda$ be a ring and assume that $W$ is a minimal cogenerator left $\Lambda$-module. Assume that $\Lambda$ is left artinian and $W$ is a finitely generated left $\Lambda$-module. Then we derive from \cite[Proposition 47.15]{wi} that $\Delta = {\rm End}_{\Lambda}(W)$ is a right artinian ring, $W$ is an injective cogenerator in Mod-$\Delta$ and $\Lambda \cong {\rm End}_{\Delta}(W)$ as rings. Therefore $\Lambda$ is a left Morita to $\Delta$. Also we conclude from \cite[Proposition 47.3]{wi} that the functor ${\rm Hom}_{\Lambda}(-,W): \Lambda$-mod$\rightarrow$ mod-$\Delta$ is a duality with the inverse duality ${\rm Hom}_{\Delta}(-,W):$ mod-$\Delta \rightarrow \Lambda$-mod. By \cite[Proposition 24.5]{an} for each finitely generated left $\Lambda$-module $X$, the lattice of all $\Lambda$-submodules of $M$ and the lattice of all $\Delta$-submodules of ${\rm Hom}_{\Lambda}(M,W)$ are anti-isomorphic. Hence for each finitely generated left $\Lambda$-module $X$, soc$({\rm Hom}_{\Lambda}(X,W)) \cong {\rm Hom}_{\Lambda}({\rm top}(X),W)$ as $\Delta$-modules.}
\end{Rem}

\section{Main results}

Let $A$ and $B$ be two $n \times 1$ and $n \times m$ matrices over $\Lambda$, respectively. A formula $\varphi(x)$ is called {\it left positive-primitive formula} over $\Lambda$ in the free variable $x$ if it has the form
$\exists \mathbf{y} ~~\begin{pmatrix}
 A & B
\end{pmatrix}(x, \mathbf{y})^t=0$, where $\mathbf{y}= (c_1, \cdots ,c_m)$ with entries in $\Lambda$ (see \cite{pe}). A left positive-primitive formula is completely determined by the matrices $A$ and $B$.  For a left $\Lambda$-module $M$, the above left positive-primitive formula define a subgroup $$\varphi(M)=\lbrace a \in M~|~\exists c_1,\cdots,c_m ~{\rm in}~M ~{\rm with}~\begin{pmatrix}
 A & B
\end{pmatrix}{\begin{pmatrix}
a & c_1 & \cdots &c_m
\end{pmatrix}}^t=0\rbrace.$$
If $f: M \rightarrow N$ is a $\Lambda$-module homomorphism, then $(a)f \in \varphi(N)$ when $a \in \varphi(M)$. This yields a covariant functor $\varphi(-): \Lambda$-{\rm mod}$ \rightarrow {\rm Ab}$ which is a subfunctor of the forgetful functor ${\rm Hom}_{\Lambda}({_\Lambda\Lambda},-)$. Note that a subfunctors $F$ of the forgetful functor ${\rm Hom}_{\Lambda}({_\Lambda\Lambda},-)$ is finitely generated if and only if $F \simeq \varphi(-)$ for some positive-primitive formula $\varphi(x)$ over $\Lambda$ (see \cite[Corollary 12.4]{pe}).

\begin{Rem}\label{r22}
{\rm Let $\Lambda$ be a ring and suppose that $R$ is the left functor ring of $\Lambda$. Then $R=\widehat{{\rm End}}_{\Lambda}(U)$, where $U= \bigoplus_{\alpha \in I}U_{\alpha}$ and $\lbrace U_{\alpha}~|~ \alpha \in I \rbrace$ is a complete set of representative of the isomorphic classes of finitely generated left $\Lambda$-modules. We define as in \cite{wi} a functor
\begin{center}
$\textbf{F}: $ {\rm Mod}$R \longrightarrow {\rm Mod}((\Lambda$-mod$)^{\rm op})$
\end{center}
as follows. For each $M \in$ Mod$R$, we set $\textbf{F}(M) = {\rm Hom}_R({\rm Hom}_{\Lambda}(-,U),M)$. It is easy to check that ${\rm Hom}_R({\rm Hom}_{\Lambda}(-,U),M):\Lambda$-mod$ \rightarrow {\rm Ab}$ is an additive covariant functor. If $f: M \rightarrow N$ is a morphism in Mod$R$, then $\textbf{F}(f) = {\rm Hom}_R({\rm Hom}_{\Lambda}(-,U),f)$ which is a natural transformation from ${\rm Hom}_R({\rm Hom}_{\Lambda}(-,U),M)$ to ${\rm Hom}_R({\rm Hom}_{\Lambda}(-,U),N)$. By \cite[Proposition 52.5]{wi}, $\textbf{F}$ is an equivalence of categories.  It is not difficult to see that the functor $\textbf{F}$ preserves and reflects simple objects, projective objects, monomorphisms, epimorphisms, essential monomorphisms and direct sums. Let $F$ be an object in Mod$((\Lambda$-mod$)^{\rm op})$. There exists a subfunctor $G$ of $F$ such that the following diagram is commutative
\begin{displaymath}
\xymatrix{
\mathbf{F}({\rm soc}(X)) \ar[dr]_{h} \ar@{->}[r]^{\mathbf{F}(\ell)} & \mathbf{F}(X) \ar[r]^{\varphi} & F  \\
&  G \ar@{^{(}->}[ur]\\
 }
\end{displaymath}
, where $\varphi$ and $h$ are isomorphisms and $\ell: {\rm soc}(X) \rightarrow X$ is the canonical injection in Mod$R$. Note that the functor $G$ is independent of choice $X$. The subfunctor $G$ of $F$ is called {\rm socle of $F$} and we denote it by ${\rm soc}(F)$.  The indecomposable projective objects with finitely generated essential multiplicity-free socle in Mod$((\Lambda$-mod$)^{\rm op})$ correspond to the indecomposable projective unitary right $R$-modules with finitely generated essential multiplicity-free socle under the equivalence $\textbf{F}$. From now on, we fix the functor $\textbf{F}$.}
\end{Rem}

A semiperfect ring $\Lambda$ is called basic if $\Lambda/J(\Lambda)$ is a direct sum of division rings \cite{an}. By \cite[Proposition 27.14]{an},  any semiperfect ring is Morita equivalent to a basic ring. A left $\Lambda$-module $M$ has {\it multiplicity-free top} if composition factors of top$(M)$ are pairwise non-isomorphic (see \cite[Sect. 1]{r}). Let $R$ be a ring with enough idempotents. We recall that $R$ is called {\it left locally noetherian {\rm (}resp., left locally artinian{\rm )}} if every finitely generated unitary left $R$-module is noetherian (resp., artinian) (see \cite[Ch. 5, Sect. 27]{wi}). Also $R$ is called {\it left locally finite} if it is both left locally artinian and left locally noetherian (see \cite[Ch. 6, Sect. 32]{wi}).  \\

Note that $R$ being left (resp., right) artinian generalized left (resp., right) QF-2 is a Morita invariant property. We are now in a position to prove our main result.

\begin{The}\label{r33}
Let $\Lambda$ be a basic ring. Then the following conditions are equivalent.
\begin{itemize}
\item[$(a)$] $\Lambda$ is a left K\"othe ring.
\item[$(b)$] $\Lambda$ is a ring of finite representation type and the Auslander ring $T$ of $\Lambda$ is a generalized right QF-2 ring.
\item[$(c)$] $\Lambda$ is a left artinian ring and $T=\widehat{{\rm End}}_{\Lambda}(V)$ is a left locally finite generalized right QF-2 ring, where $V$ is a direct sum of modules in a complete set of representative of the isomorphic classes of  finitely generated indecomposable left $\Lambda$-modules.
\item[$(d)$] $\Lambda$ is a left artinian ring and for each indecomposable left $\Lambda$-module $M$, ${\rm Hom}_{\Lambda}(M,-) \simeq \varphi(-)$ for some positive-primitive formula $\varphi(x)$ over $\Lambda$.
\item[$(e)$] $\Lambda$ is a left artinian ring and every indecomposable projective object in  ${\rm Mod}((\Lambda$-mod$)^{\rm op})$ has finitely generated essential multiplicity-free socle.
\end{itemize}
\end{The}
\begin{proof}
$(a) \Rightarrow (b)$. Let $\Lambda$ be a left K\"othe ring. By \cite[Corollary 3.2]{za}, $\Lambda$ is of finite representation type. Let $\lbrace X_1, \cdots, X_m \rbrace$ be a complete set of representative of the isomorphic classes of finitely generated indecomposable left $\Lambda$-module. $T={\rm End}_{\Lambda}(X)$ is an Auslander ring of $\Lambda$, where $X=\bigoplus_{i=1}^mX_i$. We show that $T$ is a generalized right QF-2 ring.    Let $W$ be a minimal cogenerator left $\Lambda$-module and $\Delta={\rm End}_{\Lambda}(W)$. Since $\Lambda$ is of finite representation type, $W$ is a finitely generated left $\Lambda$-module. According to Remark \ref{r2}, ${\rm Hom}_{\Lambda}(-,W): \Lambda$-mod$\rightarrow$ mod-$\Delta$ is a duality and $\Delta$ is a right artinian ring. It is easily seen that $\lbrace {\rm Hom}_{\Lambda}(X_1,W), \cdots, {\rm Hom}_{\Lambda}(X_m,W) \rbrace$ is a complete set of representative of the isomorphic classes of finitely generated indecomposable right $\Delta$-modules. Then $\Delta$ is a ring of finite representation type and $\Theta={\rm End}_{\Delta}(\bigoplus_{i=1}^m{\rm Hom}_{\Lambda}(X_i,W))$ is the Auslander ring of $\Delta$. We conclude that the Auslander ring $T$ of $\Lambda$ is isomorphic to the Auslander ring $\Theta$ of $\Delta$ because ${\rm Hom}_{\Lambda}(X,W) \cong \bigoplus_{i=1}^m{\rm Hom}_{\Lambda}(X_i,W)$ as right $\Delta$-modules. Since the property "artinian generalized right QF-2" is Morita invariant, it is enough to show that $\Theta$ is a generalized right QF-2 ring. According to Proposition \ref{r1} it is enough to show that each ${\rm Hom}_{\Lambda}(X_i,W)$ has multiplicity-free socle. Let $1 \leq l \leq m$. top$(X_l)=\bigoplus_{i=1}^t S_i$, where $t \in {\Bbb{N}}$ and each $S_i$ is a simple left $\Lambda$-module. By \cite[Corollary 3.3]{za}, $X_l$ has multiplicity-free top and so $S_i \ncong S_j$ when $i \neq j$. It follows that ${\rm Hom}_{\Lambda}(S_i, W) \ncong {\rm Hom}_{\Lambda}(S_j, W)$ when $i \neq j$. On the other hand, by Remark \ref{r2}, soc$({\rm Hom}_{\Lambda}(X_l,W)) \cong {\rm Hom}_{\Lambda}({\rm top}(X_l),W)$ as $\Delta$-modules and hence  soc$({\rm Hom}_{\Lambda}(X_l,W)) \cong \bigoplus_{i=1}^t {\rm Hom}_{\Lambda}(S_i, W)$ as $\Delta$-modules. Therefore ${\rm Hom}_{\Lambda}(X_l,W)$ has multiplicity-free socle and the result follows.\\
$(b) \Rightarrow (c)$ is clear.\\
$(c) \Rightarrow (d)$. First we show that $\Lambda$ is a ring of finite representation type. Since $\Lambda$ is left artinian, there are only finitely many non-isomorphic simple left $\Lambda$-modules. Let $\lbrace S_1, \cdots, S_n \rbrace$ be a complete set of representative of the isomorphic classes of simple left $\Lambda$-modules. Let $\lbrace V_{\alpha}~|~ \alpha \in I\rbrace$ be a complete set of representative of the isomorphic classes of  finitely generated indecomposable left $\Lambda$-modules and $T=\widehat{{\rm End}}_{\Lambda}(V)$, where $V=\bigoplus_{\alpha \in I}V_{\alpha}$. Let $\alpha \in I$. Since $V_{\alpha}$ is a finitely generated left $\Lambda$-module, there exists a non-zero $\Lambda$-module epimorphism $h: V_{\alpha} \rightarrow S_i$ for some $1 \leq i \leq n$. We know that $V$ is a generator in $\Lambda$-Mod. Hence by \cite[Proposition 51.5(1)]{wi}, $\widehat{{\rm Hom}}_{\Lambda}(V,h): \widehat{{\rm Hom}}_{\Lambda}(V,V_{\alpha}) \rightarrow \widehat{{\rm Hom}}_{\Lambda}(V,S_i)$ is non-zero. Consequently there is a non-zero $T$-module epimorphism $\varphi:  Te_{\alpha} \rightarrow M$ where $M$ is a $T$-submodule of $\widehat{{\rm Hom}}_{\Lambda}(V,S_i)$. Since by the proof of Lemma \ref{l1}, ${\rm rad}(Te_{\alpha})$ is a maximal superfluous submodule of $Te_{\alpha}$,  we can see that $\varphi({\rm rad}(Te_{\alpha}))={\rm rad}(M)$.  On the other hand, since $T$ is a left locally noetherian ring, $\widehat{{\rm Hom}}_{\Lambda}(V,S_i)$ is a noetherian left $T$-module. It follows that $M$ is a finitely generated left $T$-module and so it has a maximal submodule. This means that ${\rm rad}(M)$ is a proper submodule of $M$. Therefore the $T$-module epimorphism  $\overline{\varphi}: Te_{\alpha}/{\rm rad}(Te_{\alpha}) \rightarrow M/{\rm rad}(M)$ defined by $x + {\rm rad}(Te_{\alpha}) \mapsto (x)\varphi + {\rm rad}(M)$ is  non-zero. Since ${\rm rad}(Te_{\alpha})$ is a maximal submodule of $Te_{\alpha}$,  $Te_{\alpha}/{\rm rad}(Te_{\alpha}) \cong M/{\rm rad}(M)$ as $T$-modules. Since $T$ is a left locally finite ring, $\widehat{{\rm Hom}}_{\Lambda}(V,S_i)$ has finite length. This implies that $M/{\rm rad}(M)$ is a composition factor of $\widehat{{\rm Hom}}_{\Lambda}(V,S_i)$. By Lemma \ref{l1}, $\lbrace Te_{\alpha}/{\rm rad}(Te_{\alpha})~|~\alpha \in I \rbrace$ is a complete set of representative of the isomorphic classes of simple unitary left $T$-modules. Therefore this set is finite. Consequently $\Lambda$ is a ring of finite representation type. Next we show that every finitely generated indecomposable left $\Lambda$-module has multiplicity-free top. Let $W$ be the minimal cogenerator in $\Lambda$-Mod. $W$ is finitely generated and by Remark \ref{r2}, the functor ${\rm Hom}_{\Lambda}(-,W): \Lambda$-mod$\rightarrow$ mod-$\Delta$ is a duality with the inverse duality ${\rm Hom}_{\Delta}(-,W):$ mod-$\Delta \rightarrow \Lambda$-mod, where $\Delta={\rm End}_{\Lambda}(W)$.  Moreover for each finitely generated left $\Lambda$-module $X$, we have soc$({\rm Hom}_{\Lambda}(X,W)) \cong {\rm Hom}_{\Lambda}({\rm top}(X),W)$ as $\Delta$-modules. Let $X$ be a finitely generated indecomposable left $\Lambda$-module. Then ${\rm top}(X)=\bigoplus_{j=1}^tS'_j$, where each $S'_j$ is a simple left $\Lambda$-module. It follows that soc$({\rm Hom}_{\Lambda}(X,W)) \cong \bigoplus_{j=1}^t{\rm Hom}_{\Lambda}(S'_j,W)$ as $\Delta$-modules. We observe that each ${\rm Hom}_{\Lambda}(S'_j,W)$ is a simple right $\Delta$-module. Since $T$ is generalized right QF-2, ${\rm Hom}_{\Lambda}(S'_j,W) \ncong {\rm Hom}_{\Lambda}(S'_k,W)$ when $j \neq k$. We conclude that $X$ has multiplicity-free top. Then by \cite[Corollary 3.2]{za}, every indecomposable left $\Lambda$-module is cyclic. Let $M$ be an indecomposable left $\Lambda$-module. Then there exists a $\Lambda$-module epimorphism $f: {_\Lambda\Lambda} \rightarrow  M$. It follows that there is a monomorphism ${\rm Hom}_{\Lambda}(M,-) \rightarrow {\rm Hom}_{\Lambda}({_\Lambda\Lambda},-)$ in Mod$((\Lambda$-mod$)^{\rm op})$. Therefore ${\rm Hom}_{\Lambda}(M,-) \simeq \varphi(-)$ for some positive-primitive formula $\varphi(x)$ over $\Lambda$. \\
$(d) \Rightarrow (a)$. Let $M$ be a finitely generated indecomposable left $\Lambda$-module and assume that there exists a monomorphism $\eta: {\rm Hom}_{\Lambda}(M,-) \rightarrow {\rm Hom}_{\Lambda}({_{\Lambda}\Lambda},-)$ in  Mod$((\Lambda$-mod$)^{\rm op})$. Yoneda's Lemma implies that $\eta={\rm Hom}_{\Lambda}(f,-)$ for some $\Lambda$-module homomorphism $f: {_{\Lambda}\Lambda} \rightarrow M$. It is easy to see that $f$ is an epimorphism. This yields that there is a finite upper bound for the lengths of finitely generated indecomposable left $\Lambda$-modules and hence by \cite[Proposition 54.3]{wi}, $\Lambda$ is of finite representation type. Therefore by \cite[Corollary 4.8]{ausla2}, $\Lambda$ is a left K\"othe ring and the result follows.\\
$(c) \Rightarrow (e)$. By using of the proof of $(c) \Rightarrow (d)$, we can see that $\Lambda$ is of finite representation type and hence $T$ is the Auslander ring of $\Lambda$. Also by \cite[Corollary 4.8]{ausla2}, every left $\Lambda$-module is a direct sum of finitely generated indecomposable left $\Lambda$-modules. By Proposition \ref{r3}, there exists an additive covariant equivalence $H:$ Mod-$T \rightarrow$ Mod$R$. It is easy to see that $H$ preserves essential monomorphisms and simple modules. By Remark \ref{r22}, the functor $\mathbf{F} \circ H :$ Mod-$T \rightarrow {\rm Mod}((\Lambda$-mod$)^{\rm op})$ is an additive equivalence of categories. Let $G$ be an indecomposable projective object in Mod$((\Lambda$-mod$)^{\rm op})$. By Remark \ref{r22}, there is an indecomposable projective unitary right $R$-module $X$ such that $\mathbf{F}(X) \cong G$.  Since by the proof of \cite[Theorem 3.10]{zn}, the functor $H$ reflects indecomposable projective objects, $H(Q) \cong X$ for some indecomposable projective right $T$-module $Q$. Since $T$ is an artinian ring, $Q$ is finitely generated and by \cite[Propositions 21.3 and 31.4]{wi}, soc$(Q)$ is finitely generated and essential in $Q$. Then soc$(Q)=\bigoplus_{i=1}^nS_i$, where $n \in {\Bbb{N}}$ and each $S_i$ is a simple right $T$-module. Since $T$ is a generalized right QF-2 ring, $S_i \ncong S_j$ when $i \neq j$. On the other hand  there is an essential monomorphism $h: H({\rm soc}(Q)) \rightarrow H(Q)$. We know that $H({\rm soc}(Q)) \cong \bigoplus_{i=1}^nH(S_i)$ as $R$-modules. Therefore ${\rm soc}(H(Q))=\bigoplus_{i=1}^nT_i$ is an essential submodule of $H(Q)$, where each $T_i \cong H(S_i)$ as $R$-modules. Moreover $T_i \ncong T_j$ when $i \neq j$. Now we have an essential monomorphism $\eta: \mathbf{F}({\rm soc}(H(Q))) \rightarrow \mathbf{F}(H(Q))$ in Mod$((\Lambda$-mod$)^{\rm op})$. Since $\mathbf{F}(H(Q)) \cong G$, there is an essential monomorphism $\tau: \mathbf{F}({\rm soc}(H(Q))) \rightarrow G$ in Mod$((\Lambda$-mod$)^{\rm op})$. So ${\rm Im}(\tau)$ is an essential subfunctor of $G$ which is isomorphism to $\mathbf{F}({\rm soc}(H(Q)))$. It is easy to see that $\mathbf{F}({\rm soc}(H(Q))) \cong \bigoplus_{i=1}^n \mathbf{F}(T_i)$ and also $\mathbf{F}(T_i) \ncong \mathbf{F}(T_j)$ when $i \neq j$. Moreover each $\mathbf{F}(T_i)$ is a simple object in  Mod$((\Lambda$-mod$)^{\rm op})$. Therefore $G$ has finitely generated essential multiplicity-free socle and the result follows.\\
$(e) \Rightarrow (a)$. Let $R$ be the left functor ring of $\Lambda$. Then every indecomposable projective unitary right $R$-module has finitely generated essential multiplicity-free socle. By Proposition \ref{p6}, $\Lambda$ is a left Morita ring. Choose $W$ the minimal cogenerator left $\Lambda$-module. By \cite[Proposition 47.15(2)]{wi}, $W$ is a finitely generated left $\Lambda$-module and $\Delta={\rm End}_{\Lambda}(W)$ is a right artinian ring. We show that every finitely generated indecomposable right $\Delta$-module has multiplicity-free socle. Let $N$ be a finitely generated indecomposable right $\Delta$-module. Since $\Delta$ is right artinian, by \cite[Propositions 21.3 and 31.4]{wi}, soc$(N)$ is finitely generated and essential in $N$. It follows that soc$(N)=\bigoplus_{i=1}^sT_i$, where $s \in {\Bbb{N}}$ and each $T_i$ is a simple right $\Delta$-module. Also $E(N) \cong \bigoplus_{i=1}^sE(T_i)$ as $\Delta$-modules. We show that $T_i \ncong T_j$ when $i \neq j$. Let $\lbrace U_{\alpha}~|~\alpha \in I \rbrace$ be a complete set of representative of the isomorphic classes of  finitely generated left $\Lambda$-modules such that $R=\widehat{{\rm End}}_{\Lambda}(U)$, where $U=\bigoplus_{\alpha \in I}U_{\alpha}$. By Remark \ref{r2}, the functor ${\rm Hom}_{\Lambda}(-,W): \Lambda$-mod $\rightarrow$ mod-$\Delta$ is a duality and hence $\lbrace {\rm Hom}_{\Lambda}(U_{\alpha},W)~|~ \alpha \in I \rbrace$ is a complete set of representative of the isomorphic classes of finitely generated right $\Delta$-modules. Then $T=\widehat{{\rm End}}_{\Delta}(V)$ is a right functor ring of $\Delta$, where $V=\bigoplus_{\alpha \in I}{\rm Hom}_{\Lambda}(U_{\alpha},W)$. Since $V$ is a generator in Mod-$\Delta$ we can see that $\widehat{{\rm Hom}}_{\Delta}(V,E(N))$ is injective in Mod$T$.  Also there is an essential monomorphism $\widehat{{\rm Hom}}_{\Delta}(V,N) \rightarrow \widehat{{\rm Hom}}_{\Delta}(V,E(N))$ in Mod$T$. Thus $E(\widehat{{\rm Hom}}_{\Delta}(V,N)) \cong \bigoplus_{i=1}^s \widehat{{\rm Hom}}_{\Delta}(V,E(T_i))$ as right $T$-modules. Since by Remark \ref{r2}, $W$ is an injective cogenerator right $\Delta$-module, each $E(T_i)$ is isomorphic to a direct summand of $W$. It follows that each $E(T_i)$ is finitely generated right $\Delta$-module. Since $\widehat{{\rm Hom}}_{\Delta}(V,-): {\rm Add}(V) \rightarrow {\rm Proj}(T^{\rm op})$ is an equivalence of categories, $\widehat{{\rm Hom}}_{\Delta}(V,N)$ and each $\widehat{{\rm Hom}}_{\Delta}(V,E(T_i))$ are projective unitary right $T$-modules. Moreover ${\rm End}_T(\widehat{{\rm Hom}}_{\Delta}(V,E(T_i))) \cong {\rm End}_{\Delta}(E(T_i))$ as rings. By \cite[Proposition 19.9]{wi},  ${\rm End}_T(\widehat{{\rm Hom}}_{\Delta}(V,E(T_i)))$ is a local ring.  On the other hand, by Remark \ref{r2}, $\Lambda$ is left Morita to $\Delta$. So we deduce from \cite[Lemma 3.1]{f} that every  indecomposable projective unitary right $T$-module has finitely generated essential multiplicity-free socle. This yields that ${\rm soc}(\widehat{{\rm Hom}}_{\Delta}(V,N))=\bigoplus_{i=1}^m{T'}_i$ with ${T'}_i \ncong {T'}_j$ when $i \neq j$, where $m \in {\Bbb{N}}$ and each ${T'}_i$ is a simple unitary right $T$-module. In addition $E(\widehat{{\rm Hom}}_{\Delta}(V,N)) \cong \bigoplus_{i=1}^m E({T'}_i)$ as $T$-modules. Therefore $\bigoplus_{i=1}^s \widehat{{\rm Hom}}_{\Delta}(V,E(T_i)) \cong \bigoplus_{i=1}^mE({T'}_i)$. Since ${\rm End}_T(E({T'}_i))$ is a local ring,  $m=s$ and for each $i$, $E({T'}_i) \cong \widehat{{\rm Hom}}_{\Delta}(V,E(T_i))$ as $T$-modules.  If $T_i \cong T_j$, then $E(T_i) \cong E(T_j)$. Thus $\widehat{{\rm Hom}}_{\Delta}(V,E(T_i)) \cong \widehat{{\rm Hom}}_{\Delta}(V,E(T_j))$. It follows that $E({T'}_i) \cong E({T'}_j)$ and hence $i=j$. Consequently every finitely generated indecomposable right $\Delta$-module has multiplicity-free socle. Let $M$ be a finitely generated indecomposable left $\Lambda$-module. Since $\Lambda$ is left artinian, top$(M)=\bigoplus_{i=1}^t S_i$, where $t \in {\Bbb{N}}$ and each $S_i$ is a simple left $\Lambda$-module. By Remark \ref{r2}, ${\rm soc}({\rm Hom}_{\Lambda}(M,W)) \cong \bigoplus_{i=1}^t{\rm Hom}_{\Lambda}(S_i,W)$ as $\Delta$-modules and ${\rm Hom}_{\Lambda}(M,W)$ is a finitely generated indecomposable right $\Delta$-module. Hence  ${\rm Hom}_{\Lambda}(S_i,W) \ncong {\rm Hom}_{\Lambda}(S_j,W)$ when $i \neq j$. This implies that $S_i \ncong S_j$ when $i \neq j$. Therefore every finitely generated indecomposable left $\Lambda$-module has multiplicity-free top. Consequently by  \cite[Corollary 3.2]{za}, $\Lambda$ is a left K\"othe ring and the result follows.
\end{proof}

As a consequence, we obtain the Auslander correspondence for left Kawada rings.

\begin{The}\label{ka}
There exists a bijection between the Morita equivalence classes of left Kawada rings and the Morita equivalence classes of Auslander generalized right QF-2 rings.
\end{The}
\begin{proof}
Let $\mathfrak{A}$ be the set of Morita equivalence classes of rings of finite representation type and $\mathfrak{B}$ be the set of Morita equivalence classes of Auslander rings. By \cite[Corollary 4.7]{ausla2},  the mapping $\varphi: \mathfrak{A} \rightarrow \mathfrak{B}$ via $\Lambda \mapsto {\rm End}_{\Lambda}(U)$ is a bijection, where $U$ is a direct sum of all modules in a complete set of representative of the isomorphic classes of  finitely generated indecomposable left $\Lambda$-modules. Let $\Lambda$ be a left Kawada ring. Since $\Lambda$ is a semiperfect ring, $\Lambda$ is Morita equivalent to $\Lambda_0$ for some basic ring $\Lambda_0$. $\Lambda_0$ is a left K\"othe ring and so by Theorem \ref{r33}, $\varphi(\Lambda_0)$ is a generalized right QF-2 ring. Since $\varphi(\Lambda)$ is Morita equivalent to $\varphi(\Lambda_0)$, we can see that the Auslander ring $\varphi(\Lambda)$ is a generalized right QF-2 ring. Let $R$ be an Auslander generalized right QF-2 ring. Then there exists a representation-finite ring $\Delta$ such that $\varphi(\Delta)$ is Morita equivalent to $R$. Since $\Delta$ is semiperfect, $\Delta$ is Morita equivalent to a basic ring $\Delta_0$. It follows that $\varphi(\Delta_0)$ is Morita equivalent to $\varphi(\Delta)$. This leads that $R$ is Morita equivalent to $\varphi(\Delta_0)$. Consequently $\varphi(\Delta_0)$ is a generalized right QF-2 ring. So we deduce from Theorem \ref{r33} that $\Delta_0$ is a left K\"othe ring. Hence by \cite[Corollary 4.2]{za}, $\Delta$ is a left Kawada ring. Therefore the mapping $\varphi$ induces a bijection between the Morita equivalence classes of left Kawada rings and the Morita equivalence classes of Auslander generalized right QF-2 rings.
\end{proof}

\section*{acknowledgements}
The research of the first author was in part supported by a grant from IPM. Also, the research of the second author was in part supported by a grant from IPM (No. 1400170417).

\end{document}